\documentclass[11 pt,reqno]{amsart}

\usepackage{amsmath,amsthm,amssymb,amscd,graphicx}
\usepackage{color}
\usepackage{hyperref}
\usepackage{enumerate}
\usepackage{mathrsfs}
\usepackage{multirow}

\newtheorem{thm}{Theorem}[section]
 \newtheorem{cor}[thm]{Corollary}
 \newtheorem{lem}[thm]{Lemma}
 
 \newtheorem{prop}[thm]{Proposition}
 
 \theoremstyle{definition}
 
 \theoremstyle{remark}
 \newtheorem{rem}[thm]{Remark}
 
 \numberwithin{equation}{section}

\def\be#1 {\begin{equation} \label{#1}}
\newcommand{\ee}{\end{equation}}
\renewcommand{\phi}{\varphi}

\def\C{\mathbb C}
\def\R{\mathbb R}

\def\N{\mathbb N}

\def\H{\mathcal H}
\def\W{\mathcal W}
\def\M{\mathcal M}
\def\e{e}

\def\eps{\epsilon}
\def\dis{\displaystyle}    

 \newcommand{\s}{  \sigma }
 \newcommand{\cal}{  \mathcal   }
 
  \renewcommand{\Re}{  {\mathfrak{Re}}  }
    \renewcommand{\Im}{   {\mathfrak{Im}} }
\newcommand{\ov}{  \overline  }

\renewcommand\>{\rangle}

\definecolor{gr}{rgb}   {0.,   0.69,   0.23 }
\definecolor{NLS}{rgb}   {0.,   0.5,   1. }
\definecolor{mg}{rgb}   {0.85,  0.,    0.85}
\definecolor{yl}{rgb}   {0.8,  0.7,   0.}
\definecolor{or}{rgb}  {0.7,0.2,0.2}

\catcode`\ = \active\def {\'e} \catcode`\ = \active\def {\'E}  \catcode`\Ë = \active\def Ë{\`A}\catcode`\ = \active\def {\`e}
\catcode`\ = \active\def {\c{c}} \catcode`\ = \active\def
{\`a} \catcode`\ = \active \def {\^e} \catcode`\ = \active
\def {\`u} \catcode`\ = \active \def {\^o} \catcode`\ =
\active \def {\^u} \catcode`\ = \active \def {\^{\i}}
\catcode`\ = \active \def {\^a} \catcode`\ = \active \def
{{\"o}}

\begin{document}

\thanks{T. Chambrion is  supported by the grant "QUACO" ANR-17-CE40-0007-01.}
\thanks{L. Thomann is supported by the grants "BEKAM"  ANR-15-CE40-0001 and "ISDEEC'' ANR-16-CE40-0013.}

\author{ Thomas Chambrion }
\author{ Laurent Thomann }
\address{Thomas Chambrion, \textsc{Universit\'e de Lorraine, CNRS, INRIA, IECL, F-54000 Nancy, France}}
\email{thomas.chambrion@univ-lorraine.fr}
 \address{Laurent Thomann, \textsc{Universit\'e de Lorraine, CNRS, IECL, F-54000 Nancy, France}}
\email{laurent.thomann@univ-lorraine.fr}

\title{On the bilinear control of the Gross-Pitaevskii equation}



\begin{abstract}
In this paper we study the bilinear-control problem for the linear and non-linear Schr\"odinger equation with harmonic potential. By the means of different examples, we show how space-time smoothing effects (Strichartz estimates, Kato smoothing effect) enjoyed by the linear flow, can help to prove obstructions to controllability.
\end{abstract}

\maketitle

\section{Introduction and results} 

\subsection{Introduction}

 In this paper, for $d\geq 1$, we consider the  bi-linear control problem for the quantum harmonic oscillator
   \begin{equation} \label{NLS0}
  \left\{
      \begin{aligned}
      & i \partial_t \psi +H \psi =u(t) K(x) \psi- \sigma|\psi|^2 \psi, \qquad (t,x)\in \R \times \R^d,
       \\  & \psi(0,x)=\psi_0(x),
      \end{aligned}
    \right.
\end{equation}
 where 
    \begin{equation*}
 H=-\Delta+|x|^2=\sum_{j=1}^d \big(-\frac{\partial^2}{\partial x^2_j}+x^2_j\,\big)
 \end{equation*}
  is the harmonic oscillator, $K : \R^d \longrightarrow \R$ is a given real valued potential 
  and where the control~$u$ belongs to $L^r_{loc}(\R ; \R)$ for some $r \geq 1$.  In the 
  sequel, we will either study the case $\sigma=0$ and we will refer to this 
  equation as the  bi-linear Schr\"odinger equation, or the case $\sigma=1$ 
  (respectively $\sigma=-1$) which corresponds to the non-linear  Schr\"odinger 
  equation with a cubic   defocusing  (respectively focusing) non-linearity.
We call the linear operator $\psi \mapsto  K\psi$  the \emph{control operator}, while the (possibly non-linear) map $\psi \longmapsto i H \psi + i \sigma |\psi|^2\psi$ is usually called the \emph{drift}.   
  
   \medskip

For a given source 
$\psi_0$, the \emph{attainable set} from $\psi_0$ with controls in  
$L^r_{loc}(\R;\R)$ is the set of  $\psi_f$ for which there exist a time $T\geq 0$ 
and  a control $u$ in $L^r([0,T];\R)$ such that the solution $\psi$ 
of \eqref{NLS0} at time $T$ satisfies $\psi(T,\cdot)=\psi_f(\cdot)$. A system is \emph{controllable} in a given space $X$ if the attainable set from any point of $X$ contains $X$.  

A celebrated result \cite[Theorem 3.6]{BMS}  (see also \cite{Turinici} for the case of
the Schrdinger equation and~\cite{BCC2} for a generalization to the case of
$L^1$ controls) states         that for bi-linear equations  posed in a Banach space with 
linear drift and bounded control operator, the attainable set (from any source) with $L^r_{loc}(\R,\R)$ 
controls, $r>1$, is  contained in a countable union of compact sets.  
In an infinite dimensional Banach space, a countable union of compact sets 
is meager in Baire sense. Hence,  this result represents a deep topological obstruction 
to controllability of bi-linear control systems. Notice that this negative result does not prohibit controllability in smaller spaces, endowed with stronger norms, where the control operator is not continuous anymore. 
 
 Energy estimates have provided various obstructions to controllability of conservative equations
via a bilinear term, see \cite{BCC} for bilinear
Schrdinger  with possibly unbounded control operators
 and~\cite{ChTh1} for non-linear wave equations with
$L^1_{loc}$ controls and bounded control operators.
\medskip

Concerning the study of the well-posedness of Schr\"odinger equations with potentials, we refer to~\cite{Fujiwara, NaSte, Carles}. 

For (local) exact controllability results for NLS on a finite length interval we refer to~\cite{Beauchard1, Beauchard-Laurent,BLT, Beauchard-Laurent2}. For both the case of the bi-linear and non-linear Schr\"odinger equations, to get positive exact controllability results, the main difficulty is the choice of the ambient space. This space has to be chosen such that the equation is well-posed {\it and} the control operator is not bounded.   In~\cite{Beauchard1, Beauchard-Laurent,BLT} the fact that the control operator is not continuous is a consequence that  the Schr\"odinger equation is studied on a {\it finite length} interval with well chosen boundary conditions. Here instead, we study the equation on $\R^d$ and therefore take advantage of dispersive effects.  

For approximate controllability results  for the bi-linear Schr\"odinger equation see \cite{CMSB, MS}.

On the other hand, in the particular case  $K(x)=x$ (which does not fall in the scope of our analysis), with an explicit change of variable, one can show that the attainable set is  a finite dimensional  manifold  \cite{MR}. Notice that this result also holds for the non-linear equation,  see~\cite[Section 2.3]{ILT}.  In \cite{Nersesyan}, the authors obtained non-controllability results for the bi-linear Schr\"odinger equation on domains.

 In this note, we concentrate on control terms taking the form $u(t) K(x) \psi$, where $K$ is a  potential given once for all, and $t\mapsto u(t)$ takes real values. The extension of the results we give here to control terms with the more general form  $u(t,x) \psi$, see for instance \cite{puel}, is  beyond the scope of this work.  

 We refer to  \cite{Sarychev} for negative controllability results for non-linear  Schr\"odinger equations with additive controls. Another approach, based on Kolmogorov $\epsilon$-entropy, has been used in \cite{Shirikyan} to obtain comparable non-controllability results for the Euler equation with an additive forcing term.

 We refer to the introduction of \cite{Beauchard-Laurent} for more references on control problems and concerning results on the optimal control problem of the non-linear Schr\"odinger equation, see~\cite{HMMS}  and   \cite{Feng-Zhao, Feng-Zhao2}.

For an overview of results concerning the control of \eqref{NLS0}, see \cite{ILT}.  For an overview of controllability results of bi-linear control systems, we refer to \cite{Khapalov}.\medskip

 In the sequel, we will need the   harmonic Sobolev spaces, in other words, the Sobolev spaces based on the domain of the harmonic oscillator. For $s\geq 0$, $p\geq 1$ we define 
 \begin{equation*} 
         \W^{s, p}= \W^{s, p}(\R^d) = \big\{ f\in L^p(\R^d),\; {H}^{s/2}f\in L^p(\R^d)\big\},     
       \end{equation*}
       \begin{equation*}
           \H^{s}=   {\cal H}^{s}(\R^d) = \W^{s, 2}.
       \end{equation*}
             The natural norms are denoted by $\Vert f\Vert_{\W^{s,p}}$ and up to equivalence of norms  (see {\it e.g.} \cite[Lemma~2.4]{YajimaZhang2}), for $1<p<+\infty$,  we have
          \begin{equation}\label{equiv}
      \Vert f\Vert_{\W^{s,p}} = \Vert  H^{s/2}f\Vert_{L^{p}} \equiv \Vert (-\Delta)^{s/2} f\Vert_{L^{p}} + 
       \Vert\<x\>^{s}f\Vert_{L^{p}},
 \end{equation}
 with the notation $\<x\>=(1+|x|^{2})^{1/2}$.

\subsection{A smoothing property for the bi-linear equation}\label{sect1.2}
Consider the equation
  \begin{equation}\label{BL22}
  \left\{
      \begin{aligned}
      & i \partial_t \psi +H \psi =u(t) K(x) \psi, \qquad (t,x)\in \R \times \R^d,
       \\  & \psi(0,x)=\psi_0(x) \in \H^k(\R^d),
      \end{aligned}
    \right.
\end{equation}
in any dimension $d\geq 1$ and regularity $k \geq 0$.  Assume that $K \in \W^{k, \infty}(\R^d)$. Then for all integer $k \geq 0$, the control operator
 \begin{equation}\label{map}
 \begin{array}{rcl}
\H^k(\R^d)&\longrightarrow&\H^k(\R^d)\\[3pt]
\dis  \psi&\longmapsto &  K\psi,
 \end{array}
 \end{equation}
 is continuous (see \eqref{conti} for the proof), and therefore the general result of Ball-Marsden-Slemrod~\cite[Theorem 3.6]{BMS} applies to \eqref{BL22}. This result shows that, for fixed initial condition $ \psi_0 \in \H^k(\R^d)$, the attainable set of \eqref{BL22}
$$
\bigcup_{t  \in \R} \;\bigcup_{\substack{ u \in L_{loc}^r{(\R)}, \\  r>1}} \big\{\psi(t)\big\},
$$
is  a countable union of compact subsets of $\H^k(\R^d)$. 
 
  Our next results (Theorem~\ref{thmKnb} and Corollary~\ref{cor-compact}) give a more precise description of the attainable set of \eqref{BL22}, under the assumption $u\in L^{2}_{loc}(\R)$.

     \begin{thm}\label{thmKnb} Let  $d \geq 1$ and  $k\geq 0$ be an even integer. Let $u\in L^{2}_{loc}(\R;\R)$ and  ${K \in \W^{k+1,\infty}(\R^d,\R)}$.
 Let $\psi_0 \in \H^k(\R^d)$, then the  equation \eqref{BL22} admits a unique global solution   ${\psi \in \mathcal{C}(\R ; \H^k(\R^d))}$.  \medskip
 
 Moreover    for all $\beta <1/2$, there exists $\alpha >0$ such that
\begin{equation}\label{solu*} 
\psi(t) - \e^{itH} \psi_0  \in       \mathcal{C}^{\alpha}\big(\R; \H^{k+\beta}(\R^{d})\big),
\end{equation}
and  for all $T>0$, 
\begin{equation}\label{borne*} 
\|\psi(t) - \e^{itH} \psi_0 \|_{\mathcal{C}^{\alpha}([-T,T] ; \H^{k+\beta}(\R^d))} \leq C(T,k, \| \psi_0 \|_{ \H^k(\R^d)} ,  \|u\|_{L^2([-T,T])} ) .
\end{equation}
 \end{thm}

The proof of \eqref{borne*} relies on the Kato smoothing effect for the linear Schr\"odinger equation. It can be stated like this:  for all $\beta <1/2$ there exists $C>0$ such that for all $\phi \in L^2(\R^d)$
    \begin{equation}\label{smoothing}
 \big \| \frac1{\<x\>^{\frac12}} H^{\frac{\beta}2} e^{itH} \phi \big\|_{L^2([-2\pi,2\pi] \times \R^d)} \leq C \|\phi \|_{L^2( \R^d)}.
\end{equation}
We refer to \cite[Th\'eor\`eme 15]{poiret1} for the proof of  \eqref{smoothing}. This inequality shows that the solution of the linear Schr\"odinger  flow enjoys a gain of 1/2 derivative locally in space. \medskip

It is likely that the statement of Theorem~\ref{thmKnb} holds for any $k\in \N$, but at the price of more technicalities, therefore  in this paper we only consider the case $k\in 2\N$, which allows to work with differential operators instead of pseudo-differential operators.\medskip

The result also holds for perturbations  of $H$, namely, when $H$ is replaced with $H+W$, where~$W$ is in the Schwartz class $\mathcal{S}(\R^d; \R)$. In the  argument one has to replace $uK$ with $uK-W$. \medskip

The smoothing property stated in Theorem~\ref{thmKnb} leads to the following obstruction to controllability of equation \eqref{BL22}.

\begin{cor}\label{cor-compact} Under the assumptions of Theorem~\ref{thmKnb}, for all $\beta <1/2$, $T>0$, and $M>0$, the set 
\begin{equation*}
 \bigcup_{\substack{   t \in [-T,T] \\  \|u\|_{L^2([-T,T]; \R)} \leq M  } } \big\{ \psi(t) -e^{it H}\psi_0  \big\}
\end{equation*}
is a compact of $ \H^{k+\beta}(\R^{d})$. As a consequence,  the   set 
$$
\bigcup_{t \in \R} \;\bigcup_{ u \in L^2_{loc}(\R)}\big\{ \psi(t) -e^{it H}\psi_0\big\}
$$
is a countable union of compact subsets of $\H^{k+\beta}(\R^d)$.
\end{cor}

\begin{rem}
  With similar techniques, we can handle   the Klein-Gordon  equation (even in the non-linear case) 
 \begin{equation}\label{kg}
  \left\{
      \begin{aligned}
      & \partial_t^2\psi-\Delta\psi+m \psi=u(t)B(x) \psi-\psi^3, \quad (t,x) \in \R\times \mathcal{M}, \\
        & \psi(0,.)=\psi_0 \in H^1(\mathcal{M}),   \\
           &  \partial_t\psi(0,.)=\psi_1 \in L^2(\mathcal{M}), 
      \end{aligned}
    \right.
\end{equation}
 where $\M$ is a boundaryless compact manifold of dimension $1$ or $2$,  with $m \geq 0$ and where the potential~$B$ is assumed to be regular enough.  In this case, the result of \cite{BMS} applies, but one can additionally prove a gain of regularity, similar to Theorem \ref{thmKnb}.  Actually,  the mild solution to~\eqref{kg} reads
  \begin{equation*}
\psi(t)=S_0(t)\psi_0+S_1(t) \psi_1+\int_0^tS_1(t-s) \big(u(s)B(x) \psi(s)-\psi^3(s)\big)ds
 \end{equation*}
 where 
   \begin{equation*} 
 S_0(t)=\cos(t \sqrt{-\Delta+m} ) \;\;\text{ and }\;\; S_1(t)=\frac{\sin(t \sqrt{-\Delta+m} )}{ \sqrt{-\Delta+m}}. 
 \end{equation*}
In this context, the smoothing is realised by the gain of derivative induced by $S_1$. For non-controllability results for~\eqref{kg},  with $L^1$ controls,  we refer to \cite[Section 3]{ChTh1}.  Finally, notice that Beauchard \cite{Beauchard11} has proven a positive controllability result for the 1D bilinear wave equation with Neumann boundary conditions (this corresponds to potential with a jump after symmetrization).

\end{rem}

\subsection{Strichartz estimates and obstructions to the controllability of the non-linear equation}

The Strichartz  estimates are crucial tools in the study of the well-posedness of non-linear Schr\"odinger equation at low regularity. Let us recall them: a  couple $(q,r)\in [2,+\infty]^2$ is called admissible if 
\begin{equation*}
\frac2q+\frac{d}{r}=\frac{d}2\quad \text{and}\quad (d,q,r)\neq (2,2,+\infty).
\end{equation*}
  Then, if $(q,r)$ is an admissible couple,  for all $T>0$ there exists $C_{T}>0$ so that for all $\psi_{0}\in \H^{s}(\R^{d})$ we have  
\begin{equation}\label{Stri}
\|\e^{itH}\psi_{0}\|_{L^{q} ([-T,T],{\W}^{s,r} (\R^d)) }\leq C_{T}\|\psi_{0}\|_{\H^{s}(\R^{d}).}
\end{equation}
We will also need the inhomogeneous version of the Strichartz inequalities:  for all $ T>0 $, there exists $C_{T}>0$ so that for any  admissible couples   $ ( q_1, r_1 ) $ and     $ ( q_2, r_2 ) $ and function  $ F \in L^{q_2'}( [T,T]; \W^{s,r_2'} (\R^d)) $,
\begin{equation}\label{Stri1}
 \big\|  \int _0^t \e^{i(t-\tau)H} F(\tau) d\tau   \big\|_{L^{q_1} ([-T,T],{\W}^{s,r_1} (\R^d)) } \leq C_{T} \| F \|_{  L^{q_2'} ([-T,T],{\W}^{s,r_2'} (\R^d)) },
\end{equation}
where $ q_2' $ and $ r_2'$ are the H\"older conjugates of $ q_2 $ and $ r_2 $. We refer to \cite[Proposition 10]{poiret2} for a proof.

\subsubsection{The linear and non-linear Schr\"odinger equation in dimension $d=1$} To begin with, we consider the bi-linear Schr\"odinger equation 
   \begin{equation} \label{bilin}
  \left\{
      \begin{aligned}
      & i \partial_t \psi +H \psi =u(t) K(x) \psi , \qquad (t,x)\in \R \times \R,
       \\  & \psi(0,x)=\psi_0(x) \in \H^s(\R),
      \end{aligned}
    \right.
\end{equation}
where $K \in \H^s(\R;\R)$, for some $s\geq 0$. Then we are able to prove

\begin{thm}\label{thm1.1}
\begin{enumerate}[$(i)$]
\item Let   $K \in L^2(\R; \R)$,  $u\in L^{2}_{loc}(\R;\R)$,  and $\psi_0 \in L^2(\R;\C)$. There exists a unique global solution to  equation \eqref{bilin} in the class
 \begin{equation*} 
\psi \in   \mathcal{C}\big(\R; L^2(\R)\big) \cap  L_{loc}^4\big(\R; L^{\infty}(\R)\big). 
 \end{equation*} 
 
  This solution satisfies  
 $$\|\psi(t)\|_{L^2(\R)} =\|\psi_0\|_{L^2(\R)},\quad \forall \,t \in \R,$$
 and for all $T>0$
   \begin{equation}\label{norminf}
 \|  \psi \|_{L^4( [-T,T]; L^{\infty}(\R))}  \leq  C\big(  T,  \|\psi_0\|_{L^2(\R)} ,   \|u\|_{L^2([-T,T])}\big).
   \end{equation}  
 
Moreover, the attainable set 
$$
\bigcup_{t \in \R} \;\bigcup_{ u \in L^2_{loc}(\R;\R)} \big\{ \psi(t)\big\}
$$
is   a countable union of compact subsets of $L^2(\R)$. \medskip
\item More generally, let $s \geq 0$, $K \in \H^s(\R;\R)$,  $u\in L^{2}_{loc}(\R;\R)$ and $\psi_0 \in \H^s(\R;\C)$. Then there exists a unique global solution to  equation \eqref{bilin} in the class
 \begin{equation*} 
\psi \in   \mathcal{C}\big(\R; \H^s(\R)\big) \cap  L_{loc}^4\big(\R; \W^{s,\infty}(\R)\big).
 \end{equation*} 
  This solution satisfies  
 $$\|\psi(t)\|_{L^2(\R)} =\|\psi_0\|_{L^2(\R)},\quad \forall \,t \in \R,$$
 and for all $T>0$
      \begin{equation*} 
 \|  \psi \|_{L^{\infty}( [-T,T]; \H^{s}(\R))} +  \|  \psi \|_{L^4( [-T,T]; \W^{s,\infty}(\R))}  \leq  C\big(  T,  \|\psi_0\|_{\H^s(\R)} ,   \|u\|_{L^2([-T,T])}\big).
   \end{equation*}  ~
 
Moreover, the attainable set 
and  the attainable set 
$$
\bigcup_{t \in \R} \;\bigcup_{ u \in L^2_{loc}(\R;\R)} \big\{ \psi(t)\big\}
$$
is   a countable union of compact subsets of $\H^s(\R)$.
\end{enumerate}
\end{thm}

 This result shows that there are more obstacles than the continuity of the control operator
 \begin{equation}\label{map-prod}
 \begin{array}{rcl}
\H^s(\R)&\longrightarrow&\H^s(\R)\\[3pt]
\dis  \psi&\longmapsto &  K\psi,
 \end{array}
 \end{equation}
 for controllability  (since the map~\eqref{map-prod} is not continuous in general for a given $K \in\H^s(\R)$ when $0 \leq s \leq 1/2$). In the proof, we will crucially use the space-time Strichartz estimates to control~$K\psi$ \big(by showing that $K \psi \in L_{loc}^2\big(\R ;\H^s(\R)\big)$ when $\psi_0 \in \H^s(\R)$ and $K \in \H^s(\R)$\big) and to prove the compactness result. \medskip

Notice that for $s>1/2$, the result of Theorem~\ref{thm1.1} is a direct consequence of \cite[Theorem 3.6]{BMS}, because in this case, the map \eqref{map-prod} is continuous (see the discussion at the beginning of Section~\ref{sect1.2}).  Similarly, when $K \in \W^{1,\infty}(\R)$, then one has the strong result of Theorem~\ref{thmKnb}. The result of Theorem \ref{thm1.1} is relevant when the potential has limited regularity, namely $K \in \H^s(\R)$, $0 \leq s\leq 1/2$.\\

The previous approach also holds for the non-linear problem. Namely, consider the    cubic equation  
   \begin{equation} \label{NLS1}
  \left\{
      \begin{aligned}
      & i \partial_t \psi +H \psi =u(t) K(x) \psi -\sigma |\psi|^2\psi, \qquad (t,x)\in \R \times \R,
       \\  & \psi(0,x)=\psi_0(x) \in \H^s(\R),
      \end{aligned}
    \right.
\end{equation}
where $\sigma =\pm 1$ and $K \in \H^s(\R)$ for some  $s\geq 0$. Then we have

\begin{thm}\label{thm1.2}
Let $s \geq 0$, $K \in \H^s(\R;\R)$,  $u\in L^{2}_{loc}(\R;\R)$ and $\psi_0 \in \H^s(\R;\C)$. Then there exists a unique global solution to  equation \eqref{NLS1} in the class
 \begin{equation*} 
\psi \in   \mathcal{C}\big(\R; \H^s(\R)\big) \cap  L_{loc}^4\big(\R; \W^{s,\infty}(\R)\big).
 \end{equation*} 
   This solution satisfies  
 $$\|\psi(t)\|_{L^2(\R)} =\|\psi_0\|_{L^2(\R)},\quad \forall \,t \in \R,$$
 and for all $T>0$
      \begin{equation}\label{norminf-s}
 \|  \psi \|_{L^{\infty}( [-T,T]; \H^s(\R))}  +  \|  \psi \|_{L^4( [-T,T]; \W^{s,\infty}(\R))}  \leq  C\big(  T,  \|\psi_0\|_{\H^s(\R)} ,   \|u\|_{L^{2}([-T,T])}\big).
   \end{equation}  ~
   
Moreover, the attainable set 
$$
\bigcup_{t \in \R}\; \bigcup_{ u \in L^2_{loc}(\R;\R)} \big\{ \psi(t)\big\}
$$
is  a countable union of compact subsets of $\H^s(\R)$.
\end{thm}

This result is relevant in the sense that it shows that the non-linear term does not help to control the equation.\medskip

 All the results of this section   also hold for perturbations  of $H$, namely, when $H$ is replaced with $H+W$, where $W$ is in the Schwartz class $\mathcal{S}(\R^d; \R)$. The term $W\psi$ can be treated as a perturbation of the non-linear term.

\subsubsection{The non-linear Schr\"odinger equation in dimension $d=3$}  In order to get similar results to  Theorem~\ref{thm1.2} in higher dimension, one needs to impose more regularity on the initial condition and more regularity on the potential. This in turn will allow us to  consider a larger set of controls, namely $u \in \cup_{r>1}L_{loc}^r(\R)$ instead of $\dis u \in L^2_{loc}(\R)$, as assumed in Theorem~\ref{thm1.2}. \medskip

In this paragraph, we fix $d=3$ and we study the defocusing non-linear problem 
\begin{equation}\label{NLS}
  \left\{
      \begin{aligned}
      & i \partial_t \psi +H \psi =u(t) K(x) \psi- |\psi|^2 \psi, \qquad (t,x)\in \R \times \R^3,
       \\  & \psi(0,x)=\psi_0(x) \in \H^1(\R^3).
      \end{aligned}
    \right.
\end{equation} 

To begin with, thanks to~\eqref{Stri} and~\eqref{Stri1} we are able to state a global well-posedness result adapted to our control problem.

  \begin{prop}\label{prop1.3}  Let $u\in L^1_{loc}(\R;\R)$.
  \begin{enumerate}[$(i)$]
  \item Let   $K \in \W^{1,\infty}(\R^3;\R)$. For $\psi_0 \in \H^1(\R^3)$ the  equation~\eqref{NLS} admits a unique global solution $\psi \in \mathcal{C}(\R ; \H^1(\R^3))$.   This defines a global flow $\psi(t)=\Phi^u(t)(\psi_0)$. 
  
   \item Moreover, this solution $\psi$ satisfies the bound
 \begin{equation}  \label{bound1}
      \| \psi \|_{L^{\infty}([-T,T]; \H^1(\R^3))} \leq C( \|\psi_0\|_{\H^1(\R^3)})(1+   \|u\|_{L^1([-T,T];\R)}) ,   
       \end{equation}
       for some $C=C(  \| \psi_0 \|_{ \H^1(\R^3) })$. Furthermore,   the following bound holds true
          \begin{equation} \label{bound14}
 \|  \psi \|_{L^2([-T,T];\W^{1,6}(\R^3))}  \leq  C\big(  T,  \|\psi_0\|_{\H^1(\R^3)} ,    \|u\|_{L^1([-T,T];\R)}\big).
   \end{equation} 
         \item Let $k\geq 1$ be an integer and assume that   $K \in \W^{k,\infty}(\R^3;\R)$. Then  for $\psi_0 \in \H^k(\R^3)$  the  equation~\eqref{NLS} admits a unique global solution $\psi \in \mathcal{C}(\R ; \H^k(\R^3))$ which satisfies the bounds
 \begin{equation} \label{bound}
      \| \psi \|_{L^{\infty}([-T,T]; \H^k(\R^3))} \leq C(T,k, \| \psi_0 \|_{ \H^k(\R^3)} ,  \|u\|_{L^1([-T,T];\R)} ),   
       \end{equation}
       and 
        \begin{equation}  \label{bound15} 
 \|  \psi \|_{L^2([-T,T];\W^{k,6}(\R^3))}  \leq  C\big(  T,  \|\psi_0\|_{\H^k(\R^3)} ,   \|u\|_{L^1([-T,T];\R)}\big).
   \end{equation} 
                \end{enumerate}
            \end{prop}

The proof relies on a fixed point argument in Strichartz spaces which are well-adapted to control the non-linear term in \eqref{NLS}. Notice that  from \eqref{bound14}, we deduce that, for almost all $t \in \R$, 
   \begin{equation}\label{regpp}
    \psi(t) \in   \W^{1,6}(\R^3).
         \end{equation} 
             This is a smoothing effect for the solution,  but can not be interpreted as an obstruction to controllability of the equation~\eqref{NLS}, since the set of  times  such that \eqref{regpp} holds true depends on the control $u$. \medskip

We now state our result concerning the lack of controllability of~\eqref{NLS}

     \begin{thm}\label{BMS_NLS}
Let   $K \in \W^{1,\infty}(\R^3;\R)$ and $\psi_0 \in \H^1(\R^3)$. Denote by $\psi$ the solution of equation~\eqref{NLS} defined in Proposition~\ref{prop1.3}. Then the  attainable set 
$$
\bigcup_{t \in \R} \;\bigcup_{\substack{ u \in L_{loc}^r(\R;\R), \\  r>1}} \big\{\psi(t)\big\}
$$
is   a countable union of compact subsets of $\H^1(\R^3)$.
\end{thm}

We are able to prove similar results in dimensions   $d=1$ and $d=2$, but we do not detail them, since the proofs are similar.    The same result also holds for the bi-linear Schr\"odinger equation, but it is not relevant to state it here, since it is a direct application of  \cite[Theorem 3.6]{BMS} (see the discussion at the beginning of Section~\ref{sect1.2}). \medskip
 
Again, the results of this section   also hold for perturbations  of $H$, namely, when $H$ is replaced with $H+W$, where $W$ is in the Schwartz class $\mathcal{S}(\R^d; \R)$. The term $W\psi$ can be treated as a perturbation of the non-linear term, and the corresponding energy functional is still coercive, which is needed in our argument.

\begin{rem}
Let $k\geq 1$ be an integer. As a consequence of Proposition~\ref{prop1.3} $(iii)$ we may similarly prove that for   $K \in \W^{k,\infty}(\R^3)$ and $\psi_0 \in \H^k(\R^3)$,  the  attainable set 
$$
\bigcup_{t \in \R} \;\bigcup_{\substack{ u \in L_{loc}^r(\R), \\  r>1}} \big\{\psi(t)\big\}
$$
is   a countable union of compact subsets of $\H^k(\R^3)$. 
\end{rem}

\begin{rem}
It is worth noticing that the different results developed in this paper (excepted Corollary~\ref{cor-compact}) also hold for the Schr\"odinger equation, in the case where $H$ is replaced with $\Delta=\sum_{j=1}^d \partial^2_{x_j}$, in other words  for equations of the form
\begin{equation}\label{nls-delta}
   i \partial_t \psi +\Delta \psi =u(t) K(x) \psi +\sigma |\psi|^2 \psi.
\end{equation}
In the argument, it is enough to observe that the inequalities \eqref{smoothing}, \eqref{Stri} and \eqref{Stri1} hold true for the operator $\Delta$ (instead of $H$) and the usual Sobolev spaces $H^s(\R^d)$, $W^{s,p}(\R^d)$ (instead of $\H^s(\R^d)$, $\W^{s,p}(\R^d)$). In this setting, the conclusion of Corollary~\ref{cor-compact} is that the attainable set is meagre in the sense of Baire (the compactness is lost because the embedding $H^{s_2}(\R^d) \subset H^{s_1}(\R^d)$ is not compact, $s_1 <s_2$). \medskip

 One should be able to adapt the approach developed in \cite[Section~2.2]{ILT} (in particular \cite[Lemma~1]{ILT}) to the  equation \eqref{nls-delta}. However, the argument of \cite[Section~2.2]{ILT} does not apply to~\eqref{NLS}, because it heavily relies on the space translation invariance of the problem. 
\end{rem}

 \subsection{Notations} 

 In this paper $c,C>0$ denote constants the value of which may change
from line to line. These constants will always be universal, or uniformly bounded. For $x \in \R^d$, we write $\<x\>=(1+|x|^{2})^{1/2}$. We will sometimes use the notations $L^{p}_{T}=L^{p}([0,T])$ and $L^{p}_{T}X=L^{p}([0,T]; X)$  for $T>0$.

\section{Proof of the results concerning the bi-linear equation and the Kato smoothing effect}

\subsection{Proof of Theorem~\ref{thmKnb}}~\medskip

To begin with, we observe that it is enough to work with non-negative times, by reversibility of the Schr\"odinger equation. Therefore in the sequel we assume $t\geq 0$.\medskip

    {\it Local existence:}  We consider the map 
  \begin{equation}\label{duha0}
 \Phi(\psi)(t)=e^{it H}\psi_0-i \int_0^t u(s) e^{i(t-s) H}( K \psi)ds,
 \end{equation} 
and we will show that it is a contraction  in the space
  \begin{equation*}
B_{k,T,R}= \big\{ \|\psi\|_{L^{\infty}_T\H^k}\leq R\big\},
 \end{equation*} 
 with $R>0$ and $T>0$ to be fixed.
From the fact that $e^{it H}$ is unitary in $\H^k$ and thanks to the  Leibniz rule we deduce that 
  \begin{eqnarray}\label{borneH0}
 \| \Phi(\psi)(t)\|_{\H^k} &\leq& \|\psi_0\|_{\H^k}  +\int_0^t  |u(s)| \big\| K \psi (s)\big\|_{\H^k}ds  \nonumber \\
 &\leq& \|\psi_0\|_{\H^k} +c\big\| K  \big\|_{\W^{k,\infty}} \int_0^t  |u(s)|  \big \| \psi(s) \big \|_{\H^k}ds .
 \end{eqnarray} 
 Therefore we have 
   \begin{equation*} 
 \| \Phi(\psi)\|_{L^{\infty}_T\H^k}  
 \leq  \|\psi_0\|_{\H^k} +c\big(\int_0^T  |u(s)| ds \big) \| K  \|_{\W^{k,\infty}}  \| \psi  \|_{L_T^{\infty}\H^k}.
 \end{equation*} 
  We now choose $R =2 \|\psi_0\|_{\H^k}$ and  we fix   $T>0$ such that  $c\int_0^{T}  |u(s)| ds \leq \| K \|^{-1}_{\W^{k,\infty}}/2$. As a consequence,~$\Phi$ maps $B_{k,T,R}$ into itself. With similar estimates we can show that $\Phi$ is a contraction in $B_{k,T,R}$, namely 
    \begin{eqnarray*}
 \| \Phi(\psi_1)- \Phi(\psi_2)\|_{L^{\infty}_T\H^k}  
& \leq  & c\big(\int_0^T  |u(s)| ds \big)\| K  \|_{\W^{k,\infty}}   \| \psi_1- \psi_2\|_{{L^{\infty}_T\H^k}}\\
& \leq  & \frac12  \| \psi_1- \psi_2\|_{{L^{\infty}_T\H^k}}
  \end{eqnarray*} 
   \medskip

    {\it Global existence:} Assume that $T^{\star}>0$ is the maximal time of existence of the problem \eqref{BL22}. From the bound \eqref{borneH0}, with $\Phi(\psi)=\psi$ we deduce
      \begin{equation*} 
 \| \psi(t)\|_{\H^k} \leq  \|\psi_0\|_{\H^k} +c\big\| K  \big\|_{\W^{k,\infty}}\int_0^t  |u(s)|  \big\| \psi(s) \big\|_{\H^k}ds .
 \end{equation*} 
Therefore, by the Gr\"onwall lemma, we get that for all $t \leq T^{\star}$
   \begin{equation*} 
  \| \psi(t)\|_{\H^k} \leq   \|\psi_0\|_{\H^k} \e^{ c\| K \|_{\W^{k,\infty}}\int_0^t |u(s)|  ds   }.  
      \end{equation*}
The previous bound combined with the local existence theory implies that $T^{\star}=+\infty$. There exists a unique global solution ${\psi \in \mathcal{C}(\R ; \H^k(\R^d))}$ to \eqref{BL22}. \medskip

    {\it Proof of the smoothing effect:}    In order to prove~\eqref{solu*}, we use the Kato  smoothing effect~\eqref{smoothing}. Let $\psi$ be the solution to~\eqref{BL22} and set $\psi_1(t)=\psi(t) - \e^{itH} \psi_0 $. Then $\psi_1$ solves 
\begin{equation*} 
\psi_1(t) = -i \int_0^t u(s) e^{i(t-s) H}( K e^{isH}\psi_0)ds -i \int_0^t u(s) e^{i(t-s) H}( K \psi_1(s))ds.
\end{equation*}
Therefore
\begin{eqnarray}\label{b219} 
\|\psi_1\|_{L^{\infty}_T \H^{k+\beta}} &\leq & \|u K e^{itH}\psi_0\|_{L^1_T \H^{k+\beta}} +   \| u K \psi_1\|_{L^1_T \H^{k+\beta}} \nonumber \\
 &\leq & \|u\|_{L_T^2}  \|K e^{itH}\psi_0\|_{L^2_T \H^{k+\beta}} +    \|u\|_{L_T^2}  \| K \psi_1\|_{L^2_T \H^{k+\beta}}.
\end{eqnarray}\medskip

 $\bullet$ We write $k=2j$ with $j\geq 1$. Firstly we show that 
   \begin{equation}\label{sm1}
    \|K e^{itH}\psi_0\|_{L^2_T \H^{2j+\beta}} \leq C_T
 \end{equation}
  using the Leibniz rule:
     \begin{eqnarray} \label{eq2.5}
    \|K e^{itH}\psi_0\|_{  \H^{2j+\beta}} &= &     \|H^{j}(K e^{itH}\psi_0)\|_{ \H^{\beta}} \nonumber \\
    &\leq  &   C \sum_{\substack{ j_1,j_2,j_3 \in \N^d \\  |j_1|+|j_2|+|j_3|=2j   \\ |j_3|\neq 2j}   }   \|x^{j_1}\partial^{j_2}K \partial^{j_3}(e^{itH}\psi_0) \|_{ \H^{\beta}}\\
        &  &\qquad\qquad\qquad\qquad+\|K e^{itH}\psi_0 \|_{ \H^{\beta}}+\|K H^{j}(e^{itH}\psi_0) \|_{ \H^{\beta}}  \nonumber
 \end{eqnarray}
  where $\partial^{\ell}$ stands for derivatives in $x$ of order $|\ell|$. Each term in the sum is bounded by 
\begin{eqnarray*} 
  \|x^{j_1}\partial^{j_2}K \partial^{j_3}(e^{itH}\psi_0) \|_{ \H^{\beta}} &\leq &    \|x^{j_1}\partial^{j_2}K \partial^{j_3}(e^{itH}\psi_0) \|_{ \H^{1}} \\
  &\leq &  \|K \|_{\W^{|j_1|+|j_2|+1,\infty}}\|\psi_0\|_{\H^{j_3+1}} \\
    &\leq &  \|K \|_{\W^{2j+1,\infty}}\|\psi_0\|_{\H^{2j}} 
  \end{eqnarray*}
  thus 
  \begin{equation}\label{sm2}
  \|x^{j_1}\partial^{j_2}K \partial^{j_3}(e^{itH}\psi_0) \|_{L^2_T\H^{\beta}}  \leq CT^{1/2}.
\end{equation}
  Similarly, we have  $\|K e^{itH}\psi_0 \|_{ \H^{\beta}} \leq  \|K \|_{\W^{1,\infty}}\|\psi_0\|_{\H^{1}} $, thus 
  \begin{equation}\label{sm4}
  \|K e^{itH}\psi_0 \|_{L^2_T\H^{\beta}}  \leq CT^{1/2}.
\end{equation}
To control the contribution of the last term in \eqref{eq2.5}, we write  
  \begin{eqnarray} \label{BB1}
  \|K H^{j}(e^{itH}\psi_0)\|_{\H^{\beta}} &= &   \|Ke^{itH} (H^{j} \psi_0)\|_{\H^{\beta}} \nonumber \\
  &\leq  &  \big\| [H^{\beta/2} , K] e^{itH} (H^{j} \psi_0)\big\|_{L^2} +  \|  K H^{\beta/2}e^{itH} (H^{j} \psi_0)\|_{L^2}.
  \end{eqnarray}
  We use the commutator estimate \cite[Lemma 18]{poiret1} to get the  bound 
    \begin{equation} \label{BB2}
  \big\| [H^{\beta/2} , K] e^{itH} (H^{j} \psi_0)\big\|_{L^2}  \leq C \|\psi_0\|_{  \H^{2j}}.
    \end{equation}
By the smoothing effect~\eqref{smoothing},
  \begin{equation} \label{BB3}
 \|  K H^{\beta/2}e^{itH} (H^{j} \psi_0)\|_{L_T^2 L^2} \leq C_T,
  \end{equation}
hence by~\eqref{BB1},~\eqref{BB2} and~\eqref{BB3}
  \begin{equation}\label{sm3}
  \|K H^{j}(e^{itH}\psi_0)\|_{L^2_T\H^{\beta}}  \leq C_T.
  \end{equation}
Hence, from \eqref{sm2},  \eqref{sm4}, and \eqref{sm3} we deduce \eqref{sm1}. In the case $k=j=0$, the estimate~\eqref{sm1} is deduced from~\eqref{BB1}--\eqref{sm3}.
\medskip

 $\bullet$ We now show that  $ \|K \psi_1\|_{L^2_T \H^{k+\beta}} \leq C T^{1/2}  \|\psi_1\|_{L^\infty_T \H^{k+\beta}} $. By the fractional Leibniz rule \eqref{leibniz}, we have, for all $p> 2$
   \begin{equation} \label{conti}
 \|K \psi_1\|_{  \H^{k+\beta}} \leq   C\|K \|_{L^\infty}  \| \psi_1\|_{  \H^{k+\beta}}+ C\|K \|_{\W^{k+\beta, p}}  \| \psi_1\|_{ L^{q}},
  \end{equation}
with $q=2p/(p-2)$. For $p\gg 2$ large enough (hence $q>2$ small enough), by the Sobolev inequalities, $ \| \psi_1\|_{ L^{q}} \leq C  \| \psi_1\|_{ \H^{k+\beta}}$ which controls the first term in \eqref{conti}. To treat the second, we claim that  $\|K \|_{\W^{k+\beta, p}}  \leq C\|K \|_{\W^{k+1, \infty}}$, for $p\gg 2$ large enough. Actually, we observe that the decay $|K| \leq C \<x\>^{-k-1}$ implies that $K \in L^r(\R^d)$ for $r\gg 2$ large enough. Then one uses the interpolation inequality
$$\| K\|_{\W^{(1-\theta) s, r/\theta}(\R^d)} \leq \| K\|^{1-\theta}_{\W^{ s, \infty}(\R^d)}\| K\|^{\theta}_{L^{r}(\R^d)}, \quad 0\leq \theta \leq 1,   $$
with $s=k+1$, $\theta$ such that $(1-\theta)s=k+\beta$ and $r= \theta p$.

As a conclusion, with~\eqref{b219} we infer
$$ \|\psi_1\|_{L^{\infty}_T \H^{k+\beta}} \leq C_T+C T^{1/2}   \|u\|_{L_T^2}  \|\psi_1\|_{L^{\infty}_T \H^{k+\beta}} ,  $$
which implies, for $T>0$ small enough and which only depends on $u$ and $K$, that $ \|\psi_1\|_{L^{\infty}_T \H^{k+\beta}} \leq 2 C_T$. We are able to iterate this argument to obtain that 
    \begin{equation}\label{conti0} 
     \psi_1  \in \mathcal{C}\big(\R; \H^{k+\beta}(\R^{d})\big),
         \end{equation}
with the bound
    \begin{equation}\label{borne1*}
 \|    \psi_1 \|_{L^\infty_T \H^{k+\beta}}\leq C(T,k, \| \psi_0 \|_{ \H^k} ,  \|u\|_{L^2_T} ).
         \end{equation}
Notice that the previous estimate implies 
  \begin{equation}  \label{conti1}
\psi_1  \in  L^2\big([-T,T]; \H^{k+\beta}(\R^{d})\big).
    \end{equation}
    
Let us now show that  for all $T>0$, $\partial_t \psi \in L^2_T \H^{k-2}$, which in turn will imply that
   \begin{equation} \label{conti2} 
 \partial_t  \psi_1  \in L^2\big([-T,T]; \H^{k-2}(\R^{d})\big).
         \end{equation}
 From the equation~\eqref{BL22}, we get for all $-T\leq t\leq T$
  \begin{equation*} 
  \| \partial_t \psi (t)\|_{  \H^{k-2}} \leq   \| \psi(t) \|_{  \H^{k}}+ |u(t)| \| K\|_{\W^{k+1, \infty}} \| \psi(t) \|_{  \H^{k}},
    \end{equation*}
 thus 
  \begin{equation} \label{borne2*}
  \| \partial_t \psi \|_{L^2_T \H^{k-2}} \leq  CT^{1/2} \| \psi \|_{ L_T^{\infty} \H^{k}}+ \|u\|_{L^2_T} \| K\|_{\W^{k+1, \infty}} \| \psi \|_{ L_T^{\infty} \H^{k}},
    \end{equation}
    hence the result. 
    
    By the interpolation Lemma~\ref{lem-interp} in the appendix, applied to \eqref{conti1} and \eqref{conti2}, there exist $\alpha>0$ and $\kappa>0$ such that 
 \begin{equation}\label{conti3} 
\psi_1  \in  \mathcal{C}^{\alpha}\big([-T,T]; \H^{k+\beta-\kappa}(\R^{d})\big).
    \end{equation}
    Finally we interpolate~\eqref{conti0} and ~\eqref{conti3}, and thus, for all $\beta'<\beta$, there exists $\alpha'>0$ such that $\psi_1  \in \mathcal{C}^{\alpha'}\big([-T,T]; \H^{k+\beta'}(\R^{d})\big)$. The bound~\eqref{borne*} follows from~\eqref{borne1*},~\eqref{borne2*}, by  the interpolation argument. \bigskip


\subsection{Proof of Corollary~\ref{cor-compact}}
Fix $\psi_0 \in \H^k(\R^d)$. Let $(u_n)_{n\geq 1}$ be such that $\|u_n\|_{L^2_T} \leq K$ and consider $\psi_n$ the solution of~\eqref{BL22}  associated to~$u_n$, and let $(t_n)_{n\geq 1} \subset [-T,T]$. Set $\Psi_n(t_n)=  \psi_n(t_n) -e^{it_n H}\psi_0$.  Up to a subsequence, we can assume that $t_n \to t$ for some $t \in [-T,T]$. Let $\beta< \beta'<1/2$, then by~\eqref{borne*}, 
$$  \|\Psi_n \|_{\mathcal{C}^{\alpha}_T \H^{k+\beta'}} \leq C.$$
By the compact embedding $  \mathcal{C}^{\alpha}\big([-T,T]; \H^{k+\beta'}(\R^{d})\big) \subset   \mathcal{C} \big([-T,T]; \H^{k+\beta}(\R^{d})\big)$, there exists $\Psi  \in  \mathcal{C} \big([-T,T]; \H^{k+\beta}(\R^{d})\big)$ such that $\Psi_n \to \Psi$, up to a subsequence. Next 
\begin{eqnarray*}
\|\Psi_n(t_n)-\Psi(t)\|_{\H^{k+\beta}} &\leq & \|\Psi_n(t_n)-\Psi(t_n)\|_{\H^{k+\beta}}+ \|\Psi(t_n)-\Psi(t)\|_{\H^{k+\beta}} \nonumber  \\
&\leq &\sup_{\tau \in [-T,T]} \|\Psi_n(\tau)-\Psi(\tau)\|_{\H^{k+\beta}}+ \|\Psi(t_n)-\Psi(t)\|_{\H^{k+\beta}}.
\end{eqnarray*}
The first term in the previous line tends to 0 since $\Psi_n \to \Psi$, and the second as well, since $\Psi \in  \mathcal{C} \big([-T,T]; \H^{k+\beta}(\R^{d})\big)$.

\section{The Schr\"odinger equation in dimension $d=1$}

We prove Theorem \ref{thm1.1} and Theorem \ref{thm1.2} at the same time, namely we consider the equation 
   \begin{equation}\label{nls-proof}
  \left\{
      \begin{aligned}
      & i \partial_t \psi +H \psi =u(t) K(x) \psi -\sigma |\psi|^2\psi, \qquad (t,x)\in \R \times \R,
       \\  & \psi(0,x)=\psi_0(x) \in \H^s(\R),
      \end{aligned}
    \right.
\end{equation}
 with $\sigma=0$ or $\sigma=1$.\medskip

 {\it Local existence:} Let $\psi_0 \in \H^s(\R)$. We consider the map 
  \begin{equation*} 
 \Phi(\psi)(t)=e^{it H}\psi_0-i \int_0^t u(\tau) e^{i(t-\tau) H}( K \psi)d\tau+i \sigma \int_0^t e^{i(t-\tau) H}(|\psi|^2 \psi)d\tau,
 \end{equation*} 
and we will show that it is a contraction in some Banach space.  Let us define the Strichartz space~$X^s_T$ by the norm     $\|\psi\|_{X^s_T}= \| \psi \|_{L^{\infty}_T\H^s}+\| \psi \|_{L^{4}_T\W^{s,\infty}}$, and define  the space  
  \begin{equation*}
B_{s,T,R}= \big\{ \|\psi\|_{X^s_T}\leq R\big\},
 \end{equation*} 
 with $R>0$ and $T>0$ to be fixed.
 
By the Strichartz estimates \eqref{Stri} and \eqref{Stri1}  we get
   \begin{eqnarray*} 
 \| \Phi(\psi)\|_{X^s_T} &\leq& c \|\psi_0\|_{\H^s}+c\int_0^T \big\| |\psi|^2 \psi \big\|_{\H^s}ds+c\int_0^T  |u(\tau)| \big\| K \psi \big\|_{\H^s}d\tau  \nonumber \\
 &\leq&c \|\psi_0\|_{\H^s}+c  \big\| |\psi|^2 \psi \big\|_{L^1_T\H^s}+c \|u\|_{L^2_T} \big\| K \psi \big\|_{L^2_T\H^s}.
 \end{eqnarray*} 
 
 $\bullet$ By the generalised Leibniz rule \eqref{leibniz},
 $$\big\| |\psi|^2 \psi \big\|_{\H^s} \leq c\|  \psi \|^2_{L^\infty} \|  \psi \|_{\H^s}    , $$
 thus by the H\"older inequality
  \begin{equation}\label{temps-u}
  \big\| |\psi|^2 \psi \big\|_{L^1_T\H^s} \leq c\|  \psi \|^2_{L^2_TL^\infty} \|  \psi \|_{L^{\infty}_T\H^s}  \leq c T^{1/2}\|  \psi \|^2_{L^4_TL^\infty}\| \psi\|_{X^s_T}. 
  \end{equation}
 Since $\|  \psi \|_{L^4_TL^\infty} \leq  \| \psi\|_{X^s_T} \leq R $, we get     $\big\| |\psi|^2 \psi \big\|_{L^1_T\H^s} \leq c T^{1/2}R^3.$
 
  $\bullet$ Let us now prove that there exists $\kappa>0$ such that 
   \begin{equation}\label{claim1}
  \| K \psi \|_{L^2_T\H^s} \leq  c   T^{\kappa}\| K  \|_{\H^s}  \| \psi\|_{X^s_T}.
 \end{equation} 
In the case $s=0$ we simply write  $  \| K \psi \|_{L^2_TL^2} \leq   \| K  \|_{L^2} \|  \psi \|_{L^2_TL^\infty}   \leq T^{1/2}  \| K  \|_{L^2}  \| \psi\|_{X^0_T}$. Assume now $s>0$. By  \eqref{leibniz},
\begin{equation*}
  \| K \psi \|_{\H^s}  \leq c   \| K  \|_{\H^s} \|  \psi \|_{L^\infty} + c  \| K  \|_{L^p} \|  \psi \|_{\W^{s,q}} 
\end{equation*}
for all $2<p,q<\infty$ such that $1/p+1/q=1/2$. Then, by Sobolev, if $p>2$ is small enough, $ \| K  \|_{L^p} \leq c  \| K  \|_{\H^s}$. Finally using that $ \|  \psi \|_{L^2_T\W^{s,q}}  \leq  c\| \psi\|_{X^s_T}  $, we obtain \eqref{claim1}.\medskip

Putting the previous estimates together we have
    \begin{equation*}
 \| \Phi(\psi)\|_{X^s_T}
  \leq c \|\psi_0\|_{\H^s}+cT^{1/2} R^3 +c   T^{\kappa} \|u\|_{L^2_T}\| K  \|_{\H^s}  R .
 \end{equation*} 
  We now choose $R =2c \|\psi_0\|_{\H^s}$. Then for $T>0$ small enough, $\Phi$ maps $B_{s,T,R}$ into itself. With similar estimates we can show that $\Phi$ is a contraction in $B_{s,T,R}$, namely 
    \begin{equation*}
 \| \Phi(\psi_1)- \Phi(\psi_2)\|_{X^s_T}  
 \leq  \big[cT^{1/2} R^2 + cT^{\kappa} \|u\|_{L^2_T}\| K  \|_{\H^s} \big]  \| \psi_1- \psi_2\|_{X^s_T}.   
  \end{equation*} 
 As a conclusion there exists a unique fixed point to $\Phi$, which is a local solution to \eqref{nls-proof}. \medskip

  {\it Proof of the bound \eqref{norminf-s} for $s=0$:} Before we turn to the proof of the global existence, we prove this particular case of \eqref{norminf-s}.  The case $\psi_0 \equiv 0$ is trivial, therefore in the sequel we assume $\psi_0 \not \equiv 0$. Assume that one can solve \eqref{nls-proof} on $[0, T^{\star})$, and let $T<T^{\star}$.  Clearly, $ \|  \psi(t) \|_{L^2} = \|  \psi_0\|_{L^2}$ for all $0\leq t \leq T$. Let $0\leq t_0< T$ and $\delta >0$ such that $  t_0 +\delta \leq  T$. We have for all $0\leq t \leq \delta$
  \begin{equation*} 
 \psi(t+t_0)=e^{i t H}\psi(t_0)+i \int_{t_0}^{t_0+t} e^{i(t_0+t-\tau) H}(|\psi|^2 \psi)d\tau-i \sigma\int_{t_0}^{t_0+t} u(\tau) e^{i(t_0+t-\tau) H}(K \psi)d\tau,
 \end{equation*} 
which implies, by the Strichartz estimates \eqref{Stri} and \eqref{Stri1}
  \begin{eqnarray*} 
  \|  \psi\|_{L^4([t_0,t_0+\delta]; L^\infty)}  &\leq&c \|\psi\|_{L^{\infty}_TL^2}+c \| \psi \|^2_{L_T^{\infty}L^2}\| \psi \|_{L^{4/3}([t_0,t_0+\delta]; L^{\infty})} +c\|K\|_{L^2} \| \psi \|_{L_T^{\infty}L^2} \|u\|_{L^{4/3}([t_0,t_0+\delta])}  \\
   &\leq&c \|\psi\|_{L^{\infty}_TL^2}+c\delta^{2/3} \| \psi \|^2_{L_T^{\infty}L^2}\| \psi \|_{L^{4}([t_0,t_0+\delta]; L^{\infty})}   +cT^{1/4}\|K\|_{L^2} \| \psi \|_{L_T^{\infty}L^2} \|u\|_{L_T^{2}} \\
   &\leq&c \|\psi_0\|_{L^2}+c\delta^{2/3} \| \psi_0 \|^2_{L^2}\| \psi \|_{L^{4}([t_0,t_0+\delta]; L^{\infty})}   +cT^{1/4}\|K\|_{L^2} \| \psi_0 \|_{L^2} \|u\|_{L_T^{2}} .
 \end{eqnarray*} 
  We pick $\delta=\delta(T)>0$ such that $c\delta^{2/3} \| \psi_0 \|^2_{L^2} =1/2$ (using here that $\psi_0 \not \equiv 0$), thus the previous estimate gives
   \begin{equation*} 
  \|  \psi\|_{L^4([t_0,t_0+\delta]; L^\infty)}   \leq 2c \|\psi_0\|_{L^2}(1+  \|u\|_{L_T^{2}}   ).
   \end{equation*} 
   We write this estimate for $t_0=0, \delta, \dots, j \delta$ with $j \in \N$ such that $j \delta < T< (j+1)\delta$. We sum up and we obtain
      \begin{equation} \label{L4}
 \|  \psi \|_{L^4_T L^{\infty}}  \leq  C\big(  T,  \|\psi_0\|_{L^2} ,   \|u\|_{L_T^{2}}\big).
   \end{equation} \medskip
   
     {\it Global existence:}   Thanks to \eqref{temps-u} and \eqref{L4}, the time of existence given in the local theory only depends on   $\|\psi_0\|_{L^2}$ and    $\|u\|_{L_T^{2}}$, thus the local argument can be iterated. As a conclusion, the problem \eqref{nls-proof} is globally well-posed and one has the bound
        \begin{equation*} 
 \|  \psi \|_{L^{\infty}( [-T,T]; \H^{s}(\R))} +  \|  \psi \|_{L^4( [-T,T]; \W^{s,\infty}(\R))}  \leq  C\big(  T,  \|\psi_0\|_{\H^s(\R)} ,   \|u\|_{L^2([-T,T])}\big).
   \end{equation*} \medskip

{\it The compactness argument:} Let $u_n \rightharpoonup u$ weakly  in $L^2([0,T]; \R)$. Notice in particular that $\|u_n \|_{L^2_T} \leq C(T)$ for some $C(T)>0$. We have
  \begin{equation*} 
 \psi(t)=e^{it H}\psi_0-i \int_0^t u(\tau) e^{i(t-\tau) H}( K \psi(\tau))d\tau+i \sigma \int_0^t e^{i(t-\tau) H}(|\psi|^2 \psi)d\tau,
 \end{equation*} 
and 
  \begin{equation*} 
 \psi_n(t)=e^{it H}\psi_0-i \int_0^t u_n(\tau) e^{i(t-\tau) H}( K \psi_n(\tau))d\tau+i \sigma \int_0^t e^{i(t-\tau) H}(|\psi_n|^2 \psi_n)d\tau.
 \end{equation*} 
We set $z_n=\psi-\psi_n$, then $z_n$ satisfies 
\begin{equation}\label{zn}
z_n= \mathcal{L}(\psi,\psi_n)+\mathcal{N}(\psi,\psi_n),
\end{equation}
with 
\begin{equation*}
\mathcal{L}(\psi,\psi_n)= -i   \int_0^t \big(u(\tau)-u_n(\tau) \big) e^{i(t-\tau) H}( K \psi)d\tau  -i \int_0^t  u_n(\tau)  e^{i(t-\tau) H}\big( K (\psi-\psi_n)\big)d\tau
\end{equation*}
and 
\begin{equation*}
\mathcal{N}(\psi,\psi_n)= i \sigma \int_0^t   e^{i(t-\tau) H}\big(  (\psi-\psi_n)(\psi+\psi_n)  \ov{\psi} \big)d\tau +i \sigma \int_0^t   e^{i(t-\tau) H}\big(  (\ov{\psi}-\ov{\psi_n})\psi^2_n  \big)d\tau. 
\end{equation*}
Let us prove that  $z_n \longrightarrow 0$ in $L^{\infty}([0,T]; \H^s(\R))$.  To begin with, we state  an analogous result to \cite[Lemma~3.7]{BMS}.

\begin{lem}
Denote by 
  \begin{equation*} 
 \epsilon_n=\Big\|    \int_0^t \big(u_n(\tau) -u(\tau) \big)e^{i(t-\tau) H}( K \psi(\tau))d\tau \Big\|_{L^{\infty}_T\H^s(\R)}.
 \end{equation*} 
 Then $\eps_n \longrightarrow 0$, when $n \longrightarrow +\infty$, which completes the proof.
\end{lem}

\begin{proof}
We proceed by contradiction. Assume that there exists $\eps>0$, a subsequence of $u_n$ (still denoted by $u_n$) and a sequence $t_n \longrightarrow t \in [0,T]$ such that 
  \begin{equation} \label{contra}
 \Big\|    \int_0^{t_n} \big(u_n(\tau) -u(\tau) \big)e^{i(t_n-\tau) H}( K \psi(\tau))d\tau \Big\|_{\H^s(\R)} \geq \eps.
 \end{equation}
   Let us decompose 
 \begin{equation*}
  \Big\|  \int_0^{t_n} \big(u_n(\tau) -u(\tau) \big)e^{i(t_n-\tau) H}( K \psi(\tau))d\tau\Big\|_{\H^s(\R)}\leq \delta^1_n+ \delta^2_n+ \delta^3_n ,\end{equation*}
 with 
  \begin{equation*}
\delta^1_n= 
  \Big\|   \int_0^{t_n} \big(u_n(\tau) -u(\tau) \big) \big(   e^{i(t_n-\tau) H}- e^{i(t-\tau) H}\big)   ( K \psi(\tau))d\tau\Big\|_{\H^s(\R)},
 \end{equation*}
   \begin{equation*}
\delta^2_n= 
  \Big\|    \int_{t_n}^t \big(u_n(\tau) -u(\tau) \big)   e^{i(t-\tau) H}   ( K \psi(\tau))\Big\|_{\H^s(\R)},
 \end{equation*}
    \begin{equation*}
\delta^3_n= 
  \Big\|    \int_{0}^t \big(u_n(\tau) -u(\tau) \big)   e^{i(t-\tau) H}   ( K \psi(\tau))\Big\|_{\H^s(\R)},
 \end{equation*}
 and we will show that each of the previous terms tends to 0. This will give a contradiction to \eqref{contra}.
 
Up to a subsequence, we can assume that  for all $n\geq 1$, $t_n \leq t$ or   $t_n \geq t$. We only consider the first case, since the second is similar. By the Minkowski inequality and the unitarity of $e^{i\tau H}$
   \begin{eqnarray*} 
  \delta^1_n   &  \leq  &  \int_0^{t_n} \big|u_n(\tau) -u(\tau) \big|  \Big\|  \big(   e^{i(t_n-\tau) H}- e^{i(t-\tau) H}\big)   ( K \psi(\tau)) \Big\|_{\H^s(\R)} d\tau\\
      &  =&    \int_0^{t_n} \big|u_n(\tau) -u(\tau) \big|  \Big\|  \big(   e^{i t_n H}- e^{i t H}\big)   ( K \psi(\tau)) \Big\|_{\H^s(\R)}d\tau. 
           \end{eqnarray*} 
Then by Cauchy-Schwarz
   \begin{equation*} 
 \delta^1_n    \leq   \big\|u_n -u \big\|_{L^2_T}  \Big\|  \big(   e^{i t_n H}- e^{i t H}\big)   ( K \psi(\tau)) \Big\|_{L^2_{\tau \in [0,T]}\H^s(\R)}.
           \end{equation*} 
Now, using      \eqref{norminf-s}, observe that 
   \begin{equation}\label{borneL2}
   \| K \psi\|_{L^2_T\H^s(\R)} \leq \| K \|_{\H^s(\R)} \|  \psi\|_{L^2_T\W^{s,\infty}(\R)} <\infty.
    \end{equation}
 Hence Lemma~\ref{lemL2} below (with $d=1$ and $q=2$) applies to conclude, with the previous lines, that $   \delta^1_n   \longrightarrow 0$ when $n \longrightarrow +\infty$.  \medskip

By the Minkowski inequality, the unitarity of $e^{i\tau H}$ and the H\"older inequality
  \begin{eqnarray*}  
    \delta^2_n   &     \leq   &    \int_{t_n}^t \big|u_n(\tau) -u(\tau) \big|     \big\|   K \psi(\tau) \big\|_{\H^s(\R)}d\tau \\
  &   \leq   &     \|u_n -u \|_{L^{4/3}_{\tau \in [t_n,t]}}  \| K \psi\|_{L^4_T\H^s(\R)} \\
            &   \leq   &    |t-t_n|^{1/4} \|u_n -u \|_{L^{2}_{T}}  \| K  \|_{\H^s(\R)} \|  \psi\|_{L^4_T\W^{s,\infty}(\R)}, \label{PrT}
            \end{eqnarray*}    
where we used that $\|  \psi\|_{L^4_T\W^{s,\infty}(\R)} <\infty$ by \eqref{norminf-s}. Then,  $   \delta^2_n   \longrightarrow 0$ when $n \longrightarrow +\infty$.  \medskip

Let us now prove that $   \delta^3_n   \longrightarrow 0$ when $n \longrightarrow +\infty$.  We set $v(\tau)= e^{i(t-\tau) H}( K \psi(\tau))$. Then by~\eqref{borneL2}, $v \in L^2([0,T]; \H^s(\R))$. We expand $v$ on a Hilbertian basis $(h_k)_{k\geq 0}$ of $L^2(\R)$ (the Hermite functions for instance),
   \begin{equation*} 
v(\tau,x)=\sum_{k=0}^{+\infty} \alpha_k(\tau) h_k(x),
 \end{equation*}
 so that we have $\dis \|v(\tau, \cdot) \|^2_{\H^s}=\sum_{k=0}^{+\infty} (2k+1)^s|\alpha_k(\tau)|^2$.
 
  Let $\eta>0$, then  there exists $M>0$ large enough such that the function $g(\tau,x)=\sum_{k=0}^{M} \alpha_k(\tau) h_k(x)$ satisfies $\| v-g\|_{L^2([0,T] ; \H^s(\R))} \leq \eta/(4 \rho)$ where $\rho= \sup_{n \geq 0} \|u_n -u \|_{L^{2}_{T}} $. 

We have
   \begin{equation*} 
 \int_0^t \big(u_n(\tau) -u(\tau) \big)g(\tau)d\tau =\sum_{k=0}^{M}   h_k\int_0^t \big(u_n(\tau) -u(\tau) \big) \alpha_k(\tau)d\tau,
 \end{equation*}
 thus 
    \begin{equation*} 
\big \|\int_0^t \big(u_n(\tau) -u(\tau) \big)g(\tau)d\tau \big\|^2_{\H^s(\R)}=\sum_{k=0}^{M}  (2k+1)^s\Big|\int_0^t \big(u_n(\tau) -u(\tau) \big) \alpha_k(\tau)d\tau\Big|^2 \longrightarrow 0,
 \end{equation*}
 by the weak convergence of $(u_n)$. Finally, from the previous line, we deduce that for $n$ large enough,
   \begin{eqnarray*} 
  \delta^3_n= \Big\|   \int_0^t \big(u_n(\tau) -u(\tau) \big)v(\tau)d\tau \Big\|_{\H^s(\R)} &\leq& \frac{\eta}{4\rho}  \big\|u_n -u \big\|_{L^{2}_{T}} +  \Big\|   \int_0^t \big(u_n(\tau) -u(\tau) \big)g(\tau)d\tau \Big\|_{\H^s(\R)}\nonumber \\
 &\leq& \frac{\eta}2.
 \end{eqnarray*}
  In other words, $   \delta^3_n   \longrightarrow 0$ when $n \longrightarrow +\infty$. 
\end{proof}

We state a convergence result (slightly more general than what we need here)

\begin{lem}\label{lemL2}
Let $d\geq 1$, $2 \leq q <\infty$ and $s\geq 0$. Assume that $F \in L^q([0,T] ; \H^s(\R^d))$ and $t_n \longrightarrow t$. Then, when $n \longrightarrow +\infty$,
$$ \big\|  \big(   e^{i t_n H}- e^{i t H}\big)   F(\tau,x) \big\|_{L^q_{\tau \in [0,T]}\H^s(\R^d)}  \longrightarrow 0.    $$
\end{lem}

\begin{proof}
By unitarity of $e^{i\tau H}$, we can assume that $t=0$. Then, up to replacing $F$ by $H^{s/2}F$, it is enough to prove the result for $s=0$. We expand $F$ on the Hilbertian basis $(h_k)_{k\geq 0}$ of $L^2(\R^d)$ given by the Hermite functions: $\dis F(\tau,x) =\sum_{k=0}^{+\infty} \alpha_k(\tau) h_k(x)$. Thus
  \begin{equation}\label{borne-leb}
  \|F\|^q_{L^q_T L^2(\R^d)}=   \int_{0}^T\Big(\sum_{k=0}^{+\infty}   |\alpha_k(\tau)|^2\Big)^{q/2}d \tau <\infty.
     \end{equation}
We can write
  \begin{equation*}
  e^{it_n H}F(\tau,x) =\sum_{k=0}^{+\infty} \alpha_k(\tau) e^{i(2k+1)t_n}h_k(x),
   \end{equation*}
which gives
$$ \big\|  \big(   e^{i t_n H}-  1\big)   F(\tau,x) \big\|^2_{L^2(\R^d)}=  \sum_{k=0}^{+\infty} |e^{i(2k+1)t_n}-1|^2  |\alpha_k(\tau)|^2 ,  $$
and we conclude with the Lebesgue convergence theorem thanks to the bound
  \begin{equation*}
 \big\|  \big(   e^{i t_n H}-  1\big)   F(\tau,x) \big\|_{L^2(\R^d)} \leq 2 \Big(\sum_{k=0}^{+\infty}   |\alpha_k(\tau)|^2\Big)^{1/2} \in L^q([0,T]),
   \end{equation*}
   by \eqref{borne-leb}.
\end{proof}

By Lemma \ref{lem-prod-2}
\begin{eqnarray}
\| \mathcal{N}(\psi,\psi_n)(t) \|_{\H^s(\R)} &\leq & \int_0^t   \|   (\psi-\psi_n)(\psi+\psi_n)  \ov{\psi}  \|_{\H^s(\R)} d\tau +\int_0^t   \|  (\ov{\psi}-\ov{\psi_n})\psi^2_n  \|_{\H^s(\R)} d\tau \nonumber\\
&\leq &c \int_0^t   \|  z_n  \|_{\H^s(\R)} \big(  \|  \psi  \|^2_{\H^s(\R)  \cap \W^{s,\infty}(\R)}  + \|  \psi_n  \|^2_{\H^s(\R)  \cap \W^{s,\infty}(\R)} \big)d\tau.\label{prodw} 
 \end{eqnarray}
To simplify the exposition, we write $\mathcal{Y}^s(\R)=\H^s(\R)  \cap \W^{s,\infty}(\R)$ in the next lines. Thus, by~\eqref{zn}, \eqref{prodw}, and the inhomogeneous Strichartz estimate \eqref{Stri1} (with    $q$ and $r$ to be fixed later), for all $0 \leq t \leq T$
  \begin{equation*} 
 \|z_n(t)\|_{\H^s(\R)} \leq \eps_n +c\big\|    u_n  K z_n \big\|_{L^{q'}_t\W^{s,r'}(\R)} +c \int_0^t   \|  z_n  \|_{\H^s(\R)} \big(  \|  \psi  \|^2_{\mathcal{Y}^s(\R)} + \|  \psi_n  \|^2_{\mathcal{Y}^s(\R)}  \big)d\tau.
 \end{equation*} 
 Then by the Gr\"onwall lemma, for all $0 \leq t \leq T$ and \eqref{norminf-s}
   \begin{eqnarray} 
 \|z_n(t)\|_{\H^s(\R)} &\leq& \big(\eps_n +c\big\|    u_n  K z_n \big\|_{L^{q'}_t\W^{s,r'}(\R)}\big) \e^{ c \int_0^t    \big(  \|  \psi  \|^2_{\mathcal{Y}^s(\R)} + \|  \psi_n  \|^2_{\mathcal{Y}^s(\R)}  \big)d\tau} \nonumber\\
 &\leq&C_1(T) \big(\eps_n +c\big\|    u_n  K z_n \big\|_{L^{q'}_t\W^{s,r'}(\R)}\big) . \label{gronwall1}
 \end{eqnarray} 
 Now we claim that 
   \begin{equation}\label{claim-borne}
 \big\|    K z_n \big\|_{\W^{s,r'}(\R)} \leq c  \|K\|_{\H^s(\R)}  \|z_n\|_{\H^s(\R)} ,
   \end{equation}
    if $r$ is large enough.
 
 If $s=0$ we choose $r=\infty$ and we clearly have $\big\|    K z_n \big\|_{L^1(\R)} \leq   \|K\|_{L^2(\R)}  \|z_n\|_{L^2(\R)} $.
 
 If $s>0$, by \eqref{leibniz} we have
   \begin{equation*} 
 \big\|    K z_n \big\|_{\W^{s,r'}(\R)} \leq c  \|K\|_{\H^s(\R)}  \|z_n\|_{L^{q}(\R)} + c  \|z_n\|_{\H^s(\R)}  \|K\|_{L^{q}(\R)},
 \end{equation*} 
 with $q>2$ such that $1/2+1/q=1/r'$. Now, if $r<\infty$ is large enough, then $q>2$ is close to 2, and by the Sobolev inequality $ \|K\|_{L^{q}(\R)} \leq c \|K\|_{\H^{s}(\R)} $ and  $ \|z_n\|_{L^{q}(\R)} \leq c \|z_n\|_{\H^{s}(\R)} $, hence \eqref{claim-borne}. \medskip
 
 We come back to \eqref{gronwall1} and by \eqref{claim-borne} we get
    \begin{equation*} 
 \|z_n(t)\|_{\H^s(\R)} \leq  C_1(T) \Big(\eps_n +  \|K\|_{\H^s(\R)}  \Big(    \int_0^t |u_n(\tau)|^{q'}   \| z_n(\tau)\|^{q'}_{\H^s(\R)} d\tau \Big)^{1/{q'}}  \Big),
 \end{equation*}   
for some $1<q'<2$. Then there exists $C_2(T)>0$ such that 
  \begin{equation*} 
 \|z_n(t)\|^{q'}_{\H^s(\R)}  \leq  C_2(T) \Big(\eps^{q'}_n +    \|K\|^{q'}_{\H^s(\R)}      \int_0^t |u_n(\tau)|^{q'}   \| z_n(\tau)\|^{q'}_{\H^s(\R)}  d\tau\Big),
 \end{equation*} 
and by the Gr\"onwall lemma we get,   for all $0 \leq t \leq T$
  \begin{equation*} 
 \|z_n(t)\|_{\H^s(\R)}  \leq  C_3(T)\eps_n e^{ C_3(T)   \|K\|^{q'}_{\H^s(\R)}      \int_0^t |u_n(\tau)|^{q'} d\tau }
 \end{equation*} 
which in turn implies 
  \begin{equation*} 
 \|z_n\|_{L^{\infty}_T\H^s(\R)}  \leq  C_3(T)\eps_n e^{ C_3(T)   \|K\|^{q'}_{\H^s(\R)}      \int_0^T |u_n(\tau)|^{q'} d\tau }  \leq C_4(T) \eps_n,
 \end{equation*} 
and this latter term tends to 0, which concludes the proof. \medskip

  We now prove the last statement of Theorem~\ref{thm1.2} (the proof of Theorem~\ref{thm1.1} is similar). For fixed $\psi_0 \in \H^s(\R)$ we set 
$$K_{T,M}=\bigcup_{|t|\leq T}\bigcup_{\|u\|_{L^2(0,T)}\leq M} \big\{\psi(t)\big\}.$$ 
Let $\psi(t_j,u_j) \subset K_{T,M} $. By the reflexivity of $L^2(0,T)$, up to a subsequence, $t_j \longrightarrow t$ and $u_j  \rightharpoonup u$ weakly  in $L^2([0,T]; \R)$. Then by the previous proof, $\psi(t_j,u_j) \longrightarrow  \psi(t,u)$ in $\H^s(\R)$. As a result $K_{T,M}$ is compact in $\H^s(\R)$ and finally we can write 
$$
\bigcup_{t \in \R} \;\bigcup_{ u \in L^2_{loc}(\R;\R)} \big\{ \psi(t)\big\}= \bigcup_{T \in \N}   \bigcup_{M \in \N}    K_{T,M},
$$
as a union of compact sets.

\section{The non-linear Schr\"odinger equation in dimension $d=3$}

             \subsection{Proof of Proposition~\ref{prop1.3}}
  We first prove $(i)$.
  
 {\it Local existence:} We consider the map 
  \begin{equation}\label{duha}
 \Phi(\psi)(t)=e^{it H}\psi_0+i \int_0^t e^{i(t-s) H}(|\psi|^2 \psi)ds-i \int_0^t u(s) e^{i(t-s) H}( K \psi)ds,
 \end{equation} 
and we will show that it is a contraction in some Banach space. Namely,  we define the Strichartz norm $\|\psi\|_{X^1_T}= \| \psi \|_{L^{\infty}_T\H^1}+\| \psi \|_{L^{2}_T\W^{1,6}}$ and the space 
  \begin{equation*}
B_{T,R}= \big\{ \|\psi\|_{X^1_T}\leq R\big\},
 \end{equation*} 
 with $R>0$ and $T>0$ to be fixed.
 
By the Strichartz estimates~\eqref{Stri},~\eqref{Stri1} and the Leibniz rule
   \begin{eqnarray}\label{borneH}
 \| \Phi(\psi)\|_{X^1_T} &\leq& c \|\psi_0\|_{\H^1}+c\int_0^T \big\| |\psi|^2 \psi \big\|_{\H^1}ds +c\int_0^T  |u(s)| \big\| K \psi \big\|_{\H^1}ds  \nonumber \\
 &\leq&c \|\psi_0\|_{\H^1}+c \big\| \psi \big\|_{L_T^{\infty}\H^1}\| \psi \|^2_{L^{2}_T L^{\infty}} +c\big(\int_0^T  |u(s)| ds \big)\big\| K  \big\|_{\W^{1,\infty}} \big\| \psi \big\|_{L_T^{\infty}\H^1}.
 \end{eqnarray} 
 We now show that  there exists $\kappa>0$ such that    $\| \psi \|^2_{L^{2}_T L^{\infty}}  \leq T^{\kappa} \| \psi \|^2_{T} $. Let $0<\eps<1/2$, then the couple $(q_\eps,r_\eps)=(\frac4{1+2\eps},\frac3{1-\eps})$ is admissible and by the Sobolev inequality $\| \psi \|_{ L^{\infty}} \leq C \| \psi \|_{ \W^{1,r_\eps}} $. Then by the H\"older inequality, 
 $$  \| \psi \|^2_{L^{2}_T L^{\infty}}  \leq T^{\kappa}  \| \psi \|^2_{ L^{q_\eps}_T\W^{1,r_\eps}} \leq c T^{\kappa}   \| \psi \|^2_{X^1_T} ,$$
 for some $\kappa>0$. Thus
    \begin{equation*}
 \| \Phi(\psi)\|_{X^1_T}
  \leq c \|\psi_0\|_{\H^1}+cT^{\kappa} R^3 +cR\big(\int_0^T  |u(s)| ds \big)\big\| K  \big\|_{\W^{1,\infty}} .
 \end{equation*} 
 
 We now choose $R =4c \|\psi_0\|_{\H^1}$. Then we fix  $T_1=c_1 R^{-2/\kappa}$ with $c_1>0$ small enough such that $cT_1^{1/2} R^2 \leq 1/4$ and we fix $T_2>0$ such that  $c\int_0^{T_2}  |u(s)| ds \leq \big\| K \big\|^{-1}_{\W^{1,\infty}}/4$. Therefore, for $T=\min{(T_1,T_2)}$, $\Phi$ maps $B_{T,R}$ into itself. With similar estimates we can show that $\Phi$ is a contraction in $B_{T,R}$, namely 
    \begin{equation*}
 \| \Phi(\psi_1)- \Phi(\psi_2)\|_{X^1_T}  
 \leq  \big[cT^{\kappa} R^2 + c\big(\int_0^T  |u(s)| ds \big)\big\| K  \big\|_{\W^{1,\infty}} \big]  \| \psi_1- \psi_2\|_{X^1_T}.   
  \end{equation*} 
   \medskip

  {\it Energy bound:}   We define 
   \begin{equation*} 
       E(t)=\int_{\R^3}\big(\ov{\psi} H\psi+|\psi|^2 +\frac12|\psi|^4\big)dx= \int_{\R^3}\big( |\nabla \psi|^2 +|x|^2|\psi|^2+|\psi|^2 +\frac12|\psi|^4\big)dx.
       \end{equation*}
 Then, using that $  \partial_t \ov{\psi} =-i( H \ov{ \psi} + |\psi|^2 \ov{ \psi})+i u(t) K(x) \ov{ \psi}$, we get 
    \begin{eqnarray*} 
       E'(t) &=   & 2 \Re \int_{\R^3} \partial_t\ov{ \psi}\big(\psi+ H \psi + |\psi|^2 \psi  \big)dx \nonumber \\
       &=& -2 u(t) \Im \int_{\R^3}  K \ov{ \psi }H \psi dx \nonumber \\
              &=& 2 u(t) \Im \int_{\R^3}  \ov{ \psi }  \nabla K \cdot \nabla  \psi    dx.
                  \end{eqnarray*}
 Now we use the assumption $\nabla K \in L^\infty(\R^3)$ to get 
  \begin{equation*} 
  E'(t)  \leq C |u(t)| \|\psi\|_{L^2}  \| \nabla\psi\|_{L^2}  \leq C |u(t)| \|\psi_0\|_{L^2} E^{1/2}(t)    ,   
       \end{equation*}
  which, by integration,  implies 
   \begin{equation}\label{boundEn}
   E(t) \leq \big(E^{1/2}(0) +2C \|\psi_0\|_{L^2}  \int_0^t |u(s)|ds\big)^2.
      \end{equation}
Notice that thanks to the Sobolev inequality,   $\| \psi \|_{L^4(\R^3)} \leq C \| \psi \|_{\H^1(\R^3)}$, therefore $ E(0) \leq C(\| \psi_0 \|_{\H^1(\R^3)})$.
    \medskip
  
    {\it Global existence:} Assume that one can solve~\eqref{NLS} on $[0,T^{\star})$. By~\eqref{boundEn}, there is a   time $T_1^{\star}>0$ such that $c(T_1^{\star})^{\kappa} (R^{\star})^2 \leq 1/4$ with $R^{\star} =4c \|\psi\|_{L^{\infty}_{T^{\star}}\H^1}$. Then we fix $T^{\star}_2>0$ with  
    $$\dis {c\Big( \int_{T^{\star}-\frac{T_2^{\star}}{2}}^{T^{\star}+\frac{T_2^{\star}}{2}}  |u(s)| ds \Big) \big\| K  \big\|_{\W^{1,\infty}}  \leq 1/4}  .$$
      As a consequence, with the arguments of the local theory step, we are able to solve the equation~\eqref{NLS}, with an initial condition at $t=T^{\star}-{\min(T_1^{\star},T_2^{\star})}/{2}$, on the  time interval $[T^{\star}-{\min(T_1^{\star},T_2^{\star})}/{2},T^{\star}+{\min(T_1^{\star},T_2^{\star})}/{2}]$. This shows that the maximal solution is global in time.  \medskip
    
        {\it Proof of $(ii)$:}  
Let $0\leq \tau< T$ and $\delta >0$ such that $  \tau +\delta \leq  T$. By the Gagliardo-Nirenberg and Sobolev inequalities on $\R^3$, 
$$   \|  \psi\|_{ L^{\infty}} \leq C \|  \psi\|^{\frac12}_{  L^6}      \|  \psi\|^{\frac12}_{ \W^{1,6}} \leq C     \|  \psi\|^{\frac12}_{  \H^1}      \|  \psi\|^{\frac12}_{ \W^{1,6}}   , $$
then by the H\"older inequality
  \begin{equation}\label{holder}
  \|  \psi\|^2_{L^2([\tau,\tau+\delta]; L^{\infty})}   \leq C \delta^{1/2}    \|  \psi\|_{L^{\infty}_T \H^1}       \|  \psi\|_{L^2([\tau,\tau+\delta]; \W^{1,6})}.
 \end{equation} 
We have for all $0\leq t \leq \delta$
  \begin{equation*} 
 \psi(t+\tau)=e^{i t H}\psi(\tau)+i \int_\tau^{\tau+t} e^{i(\tau+t-s) H}(|\psi|^2 \psi)ds-i \int_\tau^{\tau+t} u(s) e^{i(\tau+t-s) H}(K \psi)ds,
 \end{equation*} 
which implies, using the same arguments  as in  \eqref{borneH}, that 
  \begin{multline*} 
  \|  \psi\|_{L^2([\tau,\tau+\delta]; \W^{1,6})}  \leq \\
    \begin{aligned}
  &\leq c \|\psi\|_{L^{\infty}_T\H^1}+c \big\| \psi \big\|_{L_T^{\infty}\H^1}\| \psi \|^2_{L^2([\tau,\tau+\delta]; L^{\infty})} +c\|K\|_{\W^{1,\infty}} \big\| \psi \big\|_{L_T^{\infty}\H^1}\big(\int_\tau^{\tau+\delta}  |u(s)| ds \big)  \\
   &\leq c \|\psi\|_{L^{\infty}_T\H^1}+c \delta^{1/2}  \big\| \psi \big\|^2_{L_T^{\infty}\H^1}   \|  \psi\|_{L^2([\tau,\tau+\delta]; \W^{1,6})}  +c  \|\psi\|_{L^{\infty}_T\H^1}  \int_0^T  |u(s)| ds,
       \end{aligned}
 \end{multline*} 
  where we used \eqref{holder}.  We pick $\delta=\delta(T)>0$ such that $c \delta^{1/2}  \big\| \psi \big\|^2_{L_T^{\infty}\H^1}  =\frac12$, thus the previous estimate gives
   \begin{equation}\label{estw16}
  \|  \psi\|_{L^2([\tau,\tau+\delta]; \W^{1,6})}  \leq 2c \|\psi\|_{L^{\infty}_T\H^1}(1+  \int_0^T  |u(s)| ds   ).
   \end{equation} 
   We write this estimate for $\tau=0, \delta, \dots, j \delta$ with $j \in \N$ such that $j \delta < T< (j+1)\delta$. We sum up and combine with \eqref{bound1}, which gives 
   \begin{equation*} 
 \|  \psi \|_{L^2_T \W^{1,6}}  \leq  C\big(  T,  \|\psi_0\|_{\H^1} ,    \int_0^T |u(s)|   ds\big),
   \end{equation*} 
which in turn implies \eqref{bound14}, thanks to \eqref{13}.\medskip

        {\it Proof of $(iii)$:}   Let $k \geq 1$,  and  let $\psi_0 \in \H^k(\R^3)$ and  $K \in \W^{k,\infty}(\R^3)$. Local existence in this case is proven as in the case $k=1$, thanks to a fixed point argument using the Strichartz norms $\|\psi\|_{X^k_T}= \| \psi \|_{L^{\infty}_T\H^k}+\| \psi \|_{L^{2}_T\W^{k,6}}$. The globalisation part is obtained as previously, since the local time of existence only depends on the energy norm and on $u$.
    
     Let us check the bound \eqref{bound}. Let $T>0$. Since $\psi$ is a fixed point in \eqref{duha}, we get for all $t \leq T$
      \begin{eqnarray*}
 \| \psi(t)\|_{\H^k} &\leq& c \|\psi_0\|_{\H^k}+c\int_0^t \big\| |\psi|^2 \psi \big\|_{\H^k}ds +c\int_0^t  |u(s)| \big\| K \psi \big\|_{\H^k}ds \\
 &\leq& c \|\psi_0\|_{\H^k}+c\int_0^t \big(\big\|   \psi(s) \big\|^2_{L^{\infty}}+|u(s)|\|K\|_{\W^{k,\infty}}\big) \| \psi(s)\|_{\H^k}ds,
 \end{eqnarray*} 
where in the previous line we used the Moser estimate \eqref{moser} to bound the non-linear term. 
Therefore, by the Gr\"onwall lemma, we get 
   \begin{eqnarray} \label{13}
  \| \psi(t)\|_{\H^k} &\leq & c \|\psi_0\|_{\H^k} \e^{ C\int_0^t \big(\|   \psi (s)\|^2_{L^{\infty}}+|u(s)| \big) ds   } \nonumber \\
  &\leq & C\big(  \|\psi_0\|_{\H^k} ,   \|  \psi \|_{L^2_TL^{\infty}} , \int_0^T |u(s)|   ds\big).
      \end{eqnarray}
 By the Sobolev inequality, from~\eqref{bound14} we deduce
  \begin{equation*} 
 \|  \psi \|_{L^2([0,T]; L^{\infty}(\R^3))}  \leq  C\|  \psi \|_{L^2([0,T];  \W^{1,6}(\R^3))} \leq C\big(  T,  \|\psi_0\|_{\H^1} ,    \int_0^T |u(s)|   ds\big),
  \end{equation*}       
which in turn, by \eqref{13}, implies \eqref{bound}.
     
     The estimate \eqref{bound15} can be obtained with similar arguments as for the special case $k=1$. We do not write the details.
     
   \subsection{Proof of Theorem~\ref{BMS_NLS}}

We adopt the strategy of Ball-Marsden-Slemrod \cite{BMS} combined with some non-linear estimates. Let $u_n \rightharpoonup  u$ weakly in $L^1([0,T]; \R)$ and fix $L \geq 0$ such that $\int_0^T|u_n(s)| ds \leq L$, $\int_0^T|u(s)| ds \leq L$. By definition of $\psi$ we have
  \begin{equation*} 
 \psi(t)=e^{it H}\psi_0-i \int_0^t u(s) e^{i(t-s) H}( K \psi)ds+i \int_0^t e^{i(t-s) H}(|\psi|^2 \psi)ds,
 \end{equation*} 
 and we define $\psi_n$
   \begin{equation*} 
 \psi_n(t)=e^{it H}\psi_0-i \int_0^t u_n(s) e^{i(t-s) H}( K \psi_n)ds+i \int_0^t e^{i(t-s) H}(|\psi_n|^2 \psi_n)ds.
 \end{equation*} 
Let us prove that $ \| \psi-\psi_n \|_{L^{\infty}_T\H^1} \longrightarrow 0$. Set $z_n=\psi-\psi_n$, then $z_n$ satisfies 
$$z_n= \mathcal{L}(\psi,\psi_n)+\mathcal{N}(\psi,\psi_n)$$
 with 
\begin{equation*}
\mathcal{L}(\psi,\psi_n)= -i \int_0^t \big(u(s)-u_n(s) \big) e^{i(t-s) H}( K \psi)ds  -i \int_0^t  u_n(s)  e^{i(t-s) H}\big( K (\psi-\psi_n)\big)ds
\end{equation*}
and 
\begin{equation*}
\mathcal{N}(\psi,\psi_n)= -i \int_0^t   e^{i(t-s) H}\big(  (\psi-\psi_n)(\psi+\psi_n)  \ov{\psi} \big)ds -i \int_0^t   e^{i(t-s) H}\big(  (\ov{\psi}-\ov{\psi_n})\psi^2_n  \big)ds. 
\end{equation*}
Since $K \in \W^{1,\infty}(\R^3)$, the map $\psi \longmapsto K\psi$ is continuous from $\H^1(\R^3)$ to $\H^1(\R^3)$ and \cite[Lemma~3.7]{BMS} applies. Thus,   when $n \longrightarrow +\infty$
\begin{equation*}
\epsilon_n := \sup_{t \in [0,T]}  \big\|  \int_0^t \big(u(s)-u_n(s) \big) e^{i(t-s) H}( K \psi)ds \big\|_{\H^1(\R^3)} \longrightarrow 0.
\end{equation*}
By Lemma~\ref{lem-prod},
\begin{eqnarray*}
\| \mathcal{N}(\psi,\psi_n)(t) \|_{\H^1(\R^3)} &\leq & \int_0^t   \|   (\psi-\psi_n)(\psi+\psi_n)  \ov{\psi}  \|_{\H^1(\R^3)} ds +\int_0^t   \|  (\ov{\psi}-\ov{\psi_n})\psi^2_n  \|_{\H^1(\R^3)} ds \\
&\leq & \int_0^t   \|  z_n  \|_{\H^1(\R^3)} \big(  \|  \psi  \|^2_{\W^{1,6}} + \|  \psi_n  \|^2_{\W^{1,6}}  \big)ds.
 \end{eqnarray*}
Therefore
\begin{equation*}
  \| z_n (t)\|_{\H^1(\R^3)}  \leq  \epsilon_n+ C\int_0^t  |u_n(s)  |   \|  z_n (s) \|_{\H^1(\R^3)}    ds +  C  \int_0^t   \|  z_n(s)  \|_{\H^1(\R^3)} \big(  \|  \psi  \|^2_{\W^{1,6}} + \|  \psi_n  \|^2_{\W^{1,6}}  \big)(s)ds,
 \end{equation*}
and by the Gr\"onwall lemma
\begin{equation*}
  \| z_n (t)\|_{\H^1(\R^3)}  \leq  \epsilon_n \exp \big( C\int_0^t  |u_n(s)  |      ds +  C    \|  \psi  \|^2_{L^2_t\W^{1,6}} + C    \|  \psi_n \|^2_{L^2_t\W^{1,6}}  \big).
 \end{equation*}
Finally, by \eqref{bound14},
\begin{equation*}
  \| z_n \|_{L^{\infty}_T\H^1}  \leq  \epsilon_n C\big(  T,  \|\psi_0\|_{\H^1} ,   L\big),
 \end{equation*}
which implies the result. \medskip
 
  The end of the proof of Theorem~\ref{BMS_NLS} relies on the same arguments as in Theorem~\ref{thm1.1}.


\appendix 
\section{Some Sobolev estimates} \label{Appendix}

In this section we gather some useful estimates in Sobolev spaces. To begin with, we have the following  generalised Leibniz rule  

\begin{lem} 
  Let $d\geq 1$ and $s\geq 0$, then the following estimates hold
\begin{equation}\label{leibniz} 
\|f\,g\|_{\W^{s,q}(\R^d)}\leq C \|f\|_{L^{q_{1}}(\R^d)}\|g\|_{\W^{s,{q'_{1}}}(\R^d)}+C \|g\|_{L^{q_{2}}(\R^d)}\|f\|_{\W^{s,{q'_{2}}}(\R^d)},
\end{equation}
~\\[-5pt]
with $1<q<\infty$, $1< q_{1},\,q_{2}\leq   \infty$ and  $1\leq  {q'_{1}},\,{q'_{2}}<  \infty$  so that 
$$\frac1q=\frac1{q_{1}}+\frac1{{q'_{1}}}=\frac1{q_{2}}+\frac1{{q'_{2}}}.$$
\end{lem}
For the proof with the usual Sobolev spaces, we refer to    \cite[Proposition 1.1, p. 105]{Taylor}. The result in our context follows by using \eqref{equiv}. Observe that in this result we must have ${q'_{1}},\,{q'_{2}}<  \infty$ and $q \neq 1, \infty$ which induces some technicalities in this paper. \medskip

 A  particular case of the previous inequality is the Moser estimate:   for $d\geq 1$ and $k \in \N$
 \begin{equation}\label{moser} 
 \|fg \|_{\H^k(\R^d)} \leq C\big(\|f \|_{L^{\infty}(\R^d)} \|g \|_{\H^k(\R^d)}+ \|g \|_{L^{\infty}(\R^d)} \|f \|_{\H^k(\R^d)}\big).
 \end{equation} ~

The following lemma will be useful 
\begin{lem}\label{lem-prod-2}
Let $s\geq 0$. There exists $c>0$ such that  for all $\phi\in \H^s(\R)$, $\chi_1 \in \H^s(\R)  \cap \W^{s,\infty}(\R)$ and $\chi_2  \in \H^s(\R)  \cap \W^{s,\infty}(\R)$
\begin{equation*}
\|  \phi \chi_1 \chi_2 \|_{\H^s(\R)} \leq c  \|  \phi  \|_{\H^s(\R)}   \|  \chi_1  \|_{\H^s(\R)  \cap \W^{s,\infty}(\R)}   \|  \chi_2  \|_{\H^s(\R)  \cap \W^{s,\infty}(\R)} .
 \end{equation*}
\end{lem}

\begin{proof}
The case $s=0$ is directly obtained by writing $\|  \phi \chi_1 \chi_2 \|_{L^2(\R)} \leq c  \|  \phi  \|_{L^2(\R)}   \|  \chi_1  \|_{L^\infty(\R) }    \|  \chi_2  \|_{L^\infty(\R) }$. 

Now we assume that $s>0$. By \eqref{leibniz} we have
\begin{equation*}
\|  \phi \chi_1 \chi_2 \|_{\H^s(\R)} \leq c   \|   \phi \|_{\H^s(\R)} \|  \chi_1  \chi_2 \|_{L^\infty(\R) }     +c\|  \phi   \|_{L^p(\R)} \|  \chi_1  \chi_2 \|_{\W^{s,q}(\R) }   
 \end{equation*}
 for all $2<p,q<\infty$ such that $1/p+1/q=1/2$. Then, by the Sobolev inequalities, if $p>2$ is small enough, $ \| \phi  \|_{L^p} \leq c  \| \phi  \|_{\H^s}$. Next, by \eqref{leibniz} again, 
 \begin{equation*}
\|   \chi_1 \chi_2 \|_{\W^{s, q}(\R)} \leq c \|\chi_1\|_{L^{q_{1}}(\R)}\|\chi_2\|_{\W^{s,{q'_{1}}}(\R)}+c \|\chi_2\|_{L^{q_{1}}(\R)}\|\chi_1\|_{\W^{s,{q'_{1}}}(\R)}  ,
 \end{equation*} 
 with  $1/q_1+1/{q'_1}=1/q$. We are able to conclude by observing that 
 $$    \|\chi\|_{L^{q_{1}}(\R)},   \|\chi\|_{\W^{s,{q'_{1}}}(\R)}  \leq \|\chi\|_{\H^{s}(\R)} +\|\chi\|_{\W^{s,{\infty}}(\R)} = \|  \chi  \|_{\H^s(\R)  \cap \W^{s,\infty}(\R)}.  $$
 \end{proof}

In the same spirit we state the following result

\begin{lem}\label{lem-prod}
There exists $c>0$ such that  for all $\phi\in \H^1(\R^3)$, $\chi_1 \in \W^{1,6}(\R^3)$ and $\chi_2 \in \W^{1,6}(\R^3)$
\begin{equation*}
\|  \phi \chi_1 \chi_2 \|_{\H^1(\R^3)} \leq c  \|  \phi  \|_{\H^1(\R^3)}   \|  \chi_1  \|_{\W^{1,6}(\R^3)}  \|  \chi_2  \|_{\W^{1,6}(\R^3)} .
 \end{equation*}
\end{lem}

\begin{proof}
From the Leibniz rule and the H\"older inequality we deduce that 
\begin{multline*}
\|  \phi \chi_1 \chi_2 \|_{\H^1(\R^3)} \leq \|  \chi_1 \chi_2  \nabla \phi \|_{L^2(\R^3)} +\|  \phi \chi_1  \nabla \chi_2    \|_{L^2(\R^3)} +\|  \phi \chi_2  \nabla \chi_1    \|_{L^2(\R^3)} +\|  \<x\> \phi  \chi_1 \chi_2  \|_{L^2(\R^3)} \\
 \leq \|  \chi_1 \|_{L^{\infty}} \|  \chi_2 \|_{L^{\infty}}  \big(\|  \nabla \phi \|_{L^2}+\|  \<x\> \phi \|_{L^2} \big) +\|  \phi     \|_{L^6} \big( \|  \chi_1     \|_{L^6}\|   \nabla \chi_2    \|_{L^6} +  \|  \chi_2     \|_{L^6}\|   \nabla \chi_1    \|_{L^6}  \big).
 \end{multline*}
 Then by the Sobolev inequalities, $ \|  \chi \|_{L^{\infty}(\R^3)} \leq C  \|  \chi  \|_{\W^{1,6}(\R^3)}$ and $\|  \phi     \|_{L^6(\R^3)} \leq C \|  \phi     \|_{\H^1(\R^3)} $, which allows to conclude.
\end{proof}\medskip

 We recall the following interpolation lemma taken from  \cite[Lemma 3.3]{BTT1}.
\begin{lem}\label{lem-interp} 
Let $T>0$ and $p\in[1,+\infty]$. Let $-\infty<\s_{2}\leq \s_{1} <+\infty$ and assume that $\psi\in  L^{p}\big([-T,T]; \H^{\s_{1}}\big)$ and $\partial_{t}\psi\in  L^{p}\big([-T,T]; \H^{\s_{2}}\big)$. Then for  all $\eps>\s_{1}/p-\s_{2}/p$, $\psi\in  L^{\infty}\big([-T,T]; \H^{\s_{1}-\eps}\big)$ and 
\begin{equation*}
\|\psi\|_{L^{\infty}_{T}\H^{\s_{1}-\eps}}\leq C\|\psi\|^{1-1/p}_{L^{p}_{T}\H^{\s_{1}}} \|\psi\|^{1/p}_{W_{T}^{1,p}\H^{\s_{2}}}.
\end{equation*}
Moreover, there exists $\eta>0$ and $\theta\in[0,1]$ so that for all $t_{1},t_{2}\in [-T,T]$
\begin{equation*} 
\|\psi(t_{1})-\psi(t_{2})\|_{\H^{\s_{1}-2\eps}}\leq C|t_{1}-t_{2}|^{\eta}\|\psi\|^{1-\theta}_{L^{p}_{T}\H^{\s_{1}}} \|  \psi\|^{\theta}_{W_{T}^{1,p}\H^{\s_{2}}}.
\end{equation*}
\end{lem}


\end{document}